[1994/12/01]
\documentclass{amsart}
\chardef\bslash=`\\ 

\newtheorem{theorem}{Theorem}[section]
\newtheorem{corollary}[theorem]{Corollary}
\newtheorem{lemma}[theorem]{Lemma}
\newtheorem{proposition}[theorem]{Proposition}

\theoremstyle{definition}

\newtheorem{remark}{Remark}[section]

\newtheorem{question}{Question}[section]

\theoremstyle{remark}

\newcommand{\B}{\mathcal{B}}

\usepackage{comment,amssymb,color}

\def\acts{\curvearrowright}

\newcommand{\R}{\mathbb{R}}

\newcommand{\N}{\mathbb{N}}
\newcommand{\Z}{\mathbb{Z}}
\newcommand{\E}{\mathbb{E}}

\newcommand{\C}{\mathcal{C}}

\newcommand{\invAlg}{\mathcal{I}}

\newcommand{\supp}{\mathrm{supp}}

\newcommand{\Meas}{\mathcal{M}(X)}
\newcommand{\invMeasG}{\mathcal{M}_G(X)}
\newcommand{\invMeasH}{\mathcal{M}_H(X)}
\newcommand{\ergMeasG}{\mathcal{M}^{\mathrm{erg}}_G(X)}
\newcommand{\ergMeasH}{\mathcal{M}^{\mathrm{erg}}_H(X)}
\newcommand{\gibMeas}{\mathcal{G}(X,\Phi)}

\newcommand{\F}{\mathcal{F}}

\newcommand{\eval}[2][\right]{\relax
  \ifx#1\right\relax \left.\fi#2#1\rvert}

\usepackage{mathptmx}

\begin{document}
\title[An SMB approach for pressure representation]{An SMB approach for pressure representation\\in amenable virtually orderable groups}
\author[Raimundo Brice\~no]{Raimundo Brice\~no}
\address{School of Mathematical Sciences\\Tel Aviv University\\Tel Aviv 69978, Israel}
\email{raimundo@alumni.ubc.ca}


\begin{abstract}
Given a countable discrete amenable virtually orderable group $G$ acting by translations on a $G$-subshift $X \subseteq S^G$ and an absolutely summable potential $\Phi$, we present a set of conditions to obtain a special integral representation of pressure $\mathrm{P}(\Phi)$. The approach is based on a Shannon-McMillan-Breiman (SMB) type theorem for Gibbs measures due to Gurevich-Tempelman (2007), and generalizes results from Gamarnik-Katz (2009), Helvik-Lindgren (2014), and Marcus-Pavlov (2015) by extending the setting to other groups besides $\Z^d$, by relaxing the assumptions on $X$ and $\Phi$, and by using sufficient convergence conditions in a mean --instead of a uniform-- sense. Under the fairly general context proposed here, these same conditions turn out to be also necessary.
\end{abstract}

\maketitle

\setcounter{tocdepth}{2}
\tableofcontents

\section{Introduction}

The pressure is a very relevant quantity that appears in the study of statistical mechanical \cite{1-georgii,1-ruelle} and dynamical \cite{1-bowen,1-keller,1-walters} systems. It has several applications in many contexts, particularly in symbolic systems, the setting of this work. Sometimes it is also known (up to a sign) as the \emph{specific Gibbs free energy} and it is a generalization of \emph{topological entropy}, another fundamental concept, specifically in symbolic dynamics \cite{1-lind}, where its definition has a combinatorial nature sometimes very difficult to handle. Consequently, it is useful to find different and useful ways to express these quantities.

Our framework and approach are the following. Given a countable discrete \emph{amenable} group $G$ with unit $e$, a \emph{$G$-subshift} $X \subseteq S^G$, and a \emph{potential} $\Phi$ on $X$, it is known that --under mild conditions on $X$ and $\Phi$-- there is a correspondence between the (\emph{$G$-invariant}) \emph{Gibbs measures} for $\Phi$ and the \emph{equilibrium states} for the \emph{local energy} function $-\varphi \in \mathcal{C}(X)$ associated to $\Phi$ (see \cite{1-ruelle,1-tempelman}). As a consequence, in this scenario the \emph{pressure} of $\Phi$ can be expressed as
\begin{equation}
\mathrm{P}(\Phi) = h(\mu) + \int{\varphi}d\mu,
\end{equation}
where $\mu$ is any Gibbs measure for $\Phi$ and $h(\mu)$ is the \emph{Kolmogorov-Sinai entropy} of $\mu$. When, in addition, $G$ is a \emph{linearly ordered group} with \emph{(algebraic) past} $G^-$, the measure-theoretic entropy takes the form
\begin{equation}
h(\mu) = \int{I_{\mu,G^-}}d\mu,
\end{equation}
where $I_{\mu,G^-}$ the \emph{information function} of the ``present'' $e$ conditioned on the past $G^-$ (this was already discussed in \cite{1-marcus} for the case $G = \Z^d$; for the general linearly ordered case, see \cite[Theorem 3.1]{1-huang}). Therefore, under all these assumptions,
\begin{equation}
\mathrm{P}(\Phi) = \int{\left(I_{\mu,G^-} + \varphi\right)}d\mu.
\end{equation}

In recent work \cite{1-gamarnik,1-marcus,1-briceno,1-adams}, there has been a progressive understanding of conditions for expressing the pressure $\mathrm{P}(\Phi)$ of certain potentials $\Phi$ on certain $G$-subshifts $X$ in a simplified particular way. The idea of the technique relies on taking the integrand $I_{\mu,G^-} + \varphi$ (or a closely related expression; see, for example, \cite{1-adams}) and having sufficient conditions so that
\begin{equation}
\label{eq4}
\mathrm{P}(\Phi) = \int{\left(I_{\mu,G^-} + \varphi\right)}d\nu,
\end{equation}
where $\nu$ is some other (or in some special cases, \emph{any other}; see \cite{1-marcus}) $G$-invariant measure supported on $X$. A especially interesting case is when $\nu$ is supported on a finite subset of $X$, because the previous integral becomes just a finite sum. One of the main motivations for having a formula like this is because it is helpful --particularly in the atomic case just mentioned-- for proving the existence of efficient algorithms for approximating $\mathrm{P}(\Phi)$, a well-known and challenging problem \cite{1-hochman,1-pavlov,1-gamarnik,1-kasteleyn,1-lieb,2-marcus,3-marcus}.

In \cite{1-gamarnik}, one of the seminal works in this direction, Gamarnik and Katz named this technique \emph{sequential cavity method} and developed it for the case $G = \Z^d$ endowed with the \emph{lexicographic order}, $\Phi$ a \emph{finite range} nearest-neigbor potential $\Phi$ for which its (unique) Gibbs measure satisfies \emph{strong spatial mixing} \cite{1-bertoin}, $X$ a nearest-neighbour \emph{shift of finite type (SFT)} \cite{1-lind} having a \emph{safe symbol} $0$ (a very strong condition implying strong topological mixing properties on $X$), and $\nu = \delta_{0^{\Z^d}}$ the atomic measure supported on the fixed --for $\Z^d$ translations-- point $0^{\Z^d} \in X$ that is $0$ everywhere. In this context, they managed to express the pressure as just the evaluation $\mathrm{P}(\Phi) = \left(I_{\mu,G^-} + \varphi\right)(0^{\Z^d})$ and some slight variations of this formula. Later, in \cite{1-marcus}, the conditions on $X$ and $\Phi$ were relaxed and replaced by less stringent ones and, maybe more importantly, Marcus and Pavlov realized that $\mathrm{P}(\Phi)$ could be expressed as the integral with respect to more general --sometimes even arbitrary-- $G$-invariant measures $\nu$, obtaining basically something identical to what we present in Eq. (\ref{eq4}), that we refer as \emph{a pressure representation}. In following work \cite{1-briceno}, we explored combinatorial aspects on $X$ for which some of the conditions presented in \cite{1-gamarnik,1-marcus} hold and, in \cite{1-adams}, we studied a variation of the pressure representation, considering a version of $I_{\mu,G^-}$ sufficient for obtaining representation and approximation results when there is a \emph{phase transition}, i.e., when there are multiple Gibbs measures associated to $\Phi$.

In this paper, we continue with the development of this body of work. As a first generalization, we move from the $\Z^d$ setting to the more general case of $G$ being a countable discrete group having an amenable orderable subgroup $H \leq G$ of finite index $[G:H]$. Secondly, we do not ask $X$ to be an SFT and, instead of finite range potentials $\Phi$, we allow potentials to have infinite range provided they are \emph{absolutely summable}. Thirdly, we show that there is great flexibility in terms of the choice of an information function --or a set of them-- in order to express $\mathrm{P}(\Phi)$, resembling the work of Helvik and Lindgren \cite{1-helvik} but in a much more general (in particular, non-abelian) setting. Finally, previous convergence conditions over certain conditional probabilities that were required to be uniform over $x \in X$ (see \cite{1-marcus,1-adams}), are now replaced by weaker ones in a mean $L^1_\nu$ sense. Moreover, in the case $H = G$ and $\nu$ ergodic, these conditions turn out to be also necessary. Part of these generalizations are based on a particular Breiman type theorem due to Gurevich and Tempelman and their study of Gibbs measures in a similar setting \cite{1-gurevich,2-gurevich,1-tempelman}.

The paper is organized as follows: In Section \ref{sec2}, we introduced the basic notions concerning amenable and orderable groups, subshifts defined on them, and Gibbs measures and their corresponding formalism. In Section \ref{sec3}, we recall the ergodic and Shannon-McMillan-Breiman (SMB) theorems for amenable groups in our setting, prove some useful facts about F{\o}lner sequences, and explain a theorem due to Gurevich and Tempelman fundamental for our work. In Section \ref{sec4}, we show a special decomposition of the SMB ratio using the order in the subgroup $H$ and state the main pressure representation theorem (Theorem \ref{pressureRep}). Finally, in Section \ref{sec5}, we compare the new results presented here to those from previous work and discuss the case $\nu = \mu$.

\section{Preliminaries: groups, subshifts, and measures}
\label{sec2}

\subsection{Amenable and orderable groups}

Let $G$ be a countable discrete (i.e., endowed with the discrete topology) group with unit $e$. Given $g \in G$ and $T,T' \subseteq G$, we write $T^{-1} = \{h^{-1}: h \in T\}$, $gT = \{gh: h \in T\}$, $Tg = \{hg: h \in T\}$, and $TT' = \{hh': h \in T, h' \in T'\}$. Denote the \emph{set of finite subsets of $T$} as
\begin{equation}
\F(T) := \{M \subseteq T: |M| < \infty\},
\end{equation}
and its \emph{subset of finite subsets intersecting $T'$} as
\begin{equation}
\label{notation}
\F_{T'}(T) := \{M \in \F(T): M \cap T' \neq \emptyset\}.
\end{equation}

A sequence $\{T_n\}$ in $\F(G)$ is {\bf left F{\o}lner (for $G$)} if $\lim_n |T_n|^{-1}|gT_n \triangle T_n| = 0$ for all $g \in G$, where $\triangle$ denotes the symmetric difference operator. We will sometimes use the \emph{little-o} notation; e.g., the previous condition will be equivalent to say that $|gT_n \triangle T_n| = o(|T_n|)$ for all $g \in G$ (recall that for $f \geq 0$ and $g > 0$, $f(n) = o(g(n))$ means that $\lim_n f(n)/g(n) = 0$). Similarly, $\{T_n\}$ is {\bf right F{\o}lner} if $\lim_n |T_n|^{-1}|T_ng \triangle T_n| = 0$, and {\bf two-sided F{\o}lner} if it is both left and right F{\o}lner. Notice that $\{T_n\}$ is left F{\o}lner if and only if $\{T_n^{-1}\}$ is right F{\o}lner. In this context, a group $G$ is called {\bf amenable} if it has a left (or, equivalently, a right) F{\o}lner sequence. It can be proven that every amenable group has a two-sided F{\o}lner sequence \cite[Chapter I.\S1, Proposition 2]{1-ornstein}, but not every left F{\o}lner sequence is right F{\o}lner. From now on, when referring to F{\o}lner sequences, we will always assume that we are talking about \underline{left} F{\o}lner sequences.

A sequence $\{T_n\}$ is said to be {\bf tempered} (or that satisfies the {\bf Shulman condition}) if $\sup_{n \in \N} |T_n|^{-1}\left|\bigcup_{k < n} T_k^{-1}T_n\right| < \infty$. It can be proven that any F{\o}lner sequence has a tempered (F{\o}lner) subsequence \cite[Proposition 1.5]{1-lindenstrauss}. Therefore, in this setting, a group $G$ is amenable if and only if it has a tempered F{\o}lner sequence.

Now, let $\leqslant$ be a total order on $G$, and $<$ the corresponding strict total order. We say that $(G,\leqslant)$ is {\bf linearly left-ordered} if $h_1 \leqslant h_2 \implies gh_1 \leqslant gh_2$ for all $g,h_1,h_2 \in G$. Similarly, $(G,\leqslant)$ is {\bf linearly right-ordered} if $h_1 \leqslant h_2 \implies h_1g \leqslant h_2g$ for all $g,h_1,h_2 \in G$, {\bf linearly bi-ordered} if it is simultaneously both linearly left- and right-ordered, and simply {\bf linearly ordered} if it is linearly left-, right-, or bi-ordered. For any linearly ordered group $G$ we can define its {\bf (algebraic) past} as the set $G^- := \{g \in G: g < e\}$, which satisfies the following properties: (1) $G^-G^- \subseteq G^-$, (2) $G^- \cap G^+ = \emptyset$, and (3) $G^- \cup \{e\} \cup G^+ = G$, where $G^+ := (G^-)^{-1}$. Moreover, in the bi-ordered case, $G^-$ also satisfies (4) $gG^-g^{-1} = G^-$ for all $g \in G$. From now on, when talking about a linearly ordered group, we will always assume that we are talking about linearly \underline{right}-ordered groups. There is no loss of generality in this assumption, because if $(G,\leqslant)$ is linearly left-ordered, then $(G,\leqslant^*)$ is linearly right-ordered, where $h_1 \leqslant^* h_2 \iff h_1^{-1} \leqslant h_2^{-1}$. An {\bf orderable} group $G$ will be any group that admits a total order $\leqslant$ such that $(G,\leqslant)$ is linearly right- (or equivalently, left-) ordered. A good account of the role of orderings in dynamics and group theory can be found in \cite{1-deroin}.

A partial order $\leqslant$ on $G$ is said to be {\bf locally invariant} if for all $g_1, g_2 \in G$ with $g_2 \neq e$, we have either $g_1 \leqslant g_1g_2$ or $g_1 \leqslant g_1g_2^{-1}$. Clearly, if $(G,\leqslant)$ is a linearly ordered, then $\leqslant$ is locally invariant. In the amenable case, we also have the converse.

\begin{theorem}[{\cite{1-linnell}}]
Every amenable group with a locally invariant partial order is orderable.
\end{theorem}

In other words, the results discussed in this paper will also apply in the a priori more general case of amenable groups with locally invariant partial orders.

Given $H \subseteq G$, we denote by $H \leq G$ whenever $H$ is a {\bf subgroup} of $G$. For $g \in G$, we call $gH$ and $Hg$ the {\bf left coset} and {\bf right coset}, respectively. The {\bf index} of $H$ in $G$ is the cardinality of the number of different left (or right) cosets and it will be denoted by $[G:H]$. In this paper, we will deal with \underline{left} cosets and subgroups of finite index. A {\bf (left) transversal} will be any subset $K \subseteq G$ containing exactly one element from each (left) coset. In this case, we can write $G = KH$ and $kH \cap k'H = \emptyset$ for all $k,k' \in K$ such that $k \neq k'$.

Given a property $\mathrm{P}$ (e.g., being amenable or being orderable), a group $G$ is said to be {\bf virtually} $\mathrm{P}$ if there is a finite index subgroup $H \leq G$ such that $H$ has property $\mathrm{P}$. It is well-known that virtually amenable groups are amenable \cite[Corollary 4.5.8]{1-ceccherini}. On the other hand, virtually orderable groups are a strictly larger class than orderable groups. It suffices to consider the direct product $G \times N$ of an orderable group $G$ (with unit $e_G$) with a nontrivial finite group $N$ (with unit $e_N$) to obtain a virtually orderable group which is not orderable, since $(e_G,n) \in G \times N$ is a {\bf torsion element} for all $n \in N$, i.e., there exists $m \in \N$ such that $(e_G,n)^m = (e_G,e_N)$, and this is an obstruction to orderability \cite[Section 1.4.1]{1-deroin}. We say that a group is {\bf torsion-free} if the only torsion element is the identity $e$.

Examples of amenable orderable groups include all the discrete \emph{torsion-free abelian groups} (e.g., $\Z^d$ and $\mathbb{Q}^d$ endowed with the discrete topology) and, more generally, all the discrete \emph{torsion-free nilpotent groups} (e.g., the discrete Heisenberg group $H_3(\Z)$); see \cite{1-deroin}.

\subsection{$G$-subshifts and $G$-invariant measures}

Given a countable group $G$, a finite set $S$, consider $S^G = \{x: G \to S\}$ endowed with the product topology and the \underline{left} action $G \acts S^G$ by translations where, for an element $g \in G$ and a {\bf point} $x \in S^G$, the point $g \cdot x \in S^G$ is defined as $(g \cdot x)(h) = x(hg)$ for $h \in G$. Given $X \subseteq S^G$, we write $g \cdot X = \{g \cdot x: x \in X\}$ and say that $X$ is {\bf $G$-invariant} if $g \cdot X = X$ for all $g \in G$. A subset $X \subseteq S^G$ is a {\bf $G$-subshift} if it is a compact and $G$-invariant set. In such case, $X$ will be said to be a {\bf shift of finite type (SFT)} if there exists $M \in \F(G)$ and $\mathcal{P} \subseteq S^M$ such that
\begin{equation}
X = \left\{x \in S^G: (g \cdot x)_M \notin \mathcal{P} \text{ for all } g \in G\right\},
\end{equation}
where, for $T \subseteq G$ (finite or infinite), $x_T$ denotes the restriction $x_T: T \to S$ of the point $x$ to $T$. In addition, we denote by $[x_T] = \{y \in X: y_T = x_T\}$ the corresponding cylinder set, $X_T = \{x_T: x \in X\}$ the set of restrictions to $T$, and $\B_T$ the $\sigma$-algebra generated by $\{[x_K]: x \in X, K \in \F(T)\}$, where $\B_G$ corresponds to the {\bf Borel $\sigma$-algebra}. Notice that for $g \in G$, $T \subseteq G$, and $x \in X$,
\begin{equation}
g \cdot [x_T] = \bigcap_{h \in T} \{g \cdot y: y(h) = x(h)\} = \bigcap_{h \in Tg^{-1}} \{z: z(h) = (g \cdot x)(h)\} = [(g \cdot x)_{Tg^{-1}}].
\end{equation}

Whenever $T$ is a singleton $\{h\}$, we will omit the brackets and write $x_h$, $\F_h(T)$, etc.

Let $X$ be a $G$-subshift and $\Meas$ the set of all Borel probability measures $\nu$ on $X$. Given a subgroup $H \leq G$, we define $\invAlg_H := \{A \in \B_G: h \cdot A = A \text{ for all } h \in H\}$, the {\bf $\sigma$-algebra of $H$-invariant sets} in $\B_G$. We denote by $\invMeasH$ the set of {\bf $H$-invariant measures} in $\Meas$, i.e., the ones that satisfy $\nu(A) = \nu(h \cdot A)$ for all $h \in H$ and $A \in \B_G$. The set $\ergMeasH$ of {\bf $H$-ergodic measures} will be the one containing the measures $\nu \in \invMeasH$ for which $\invAlg_H$ is \emph{trivial}, i.e., $\nu(A) \in \{0,1\}$ for all $A \in \invAlg_H$.

Given $\nu \in \Meas$, we denote by $L^1_\nu$ the Banach space of real-valued measurable functions $f: X \to \R$ --or more precisely, classes of functions $\nu$-a.e. equally valued-- such that $\|f\|_{\nu} := \int{|f|}d\nu < \infty$. In addition, for $T \subseteq G$, we define the {\bf $T$-support} of $\nu$ as
\begin{equation}
\supp(\nu,T) := \left\{x \in X: \nu([x_M]) > 0 \text{ for all } M \in \F(T)\right\}.
\end{equation}

Notice that if $T_1 \subseteq T_2 \subseteq G$, then $X \supseteq \supp(\nu,T_1) \supseteq \supp(\nu,T_2)$. We will sometimes abbreviate the $G$-support of $\nu$ by $\supp(\nu)$.

\subsection{Gibbs measures}

Given a $G$-subshift $X$, consider now $\Phi: \F(G) \times X \to \R$ an {\bf (absolutely summable) potential}, i.e., a function $\Phi$ that satisfies the following properties:
\begin{enumerate}
\item $\Phi(M,x) = \Phi(M,y)$ for all $x,y \in X$ and $M \in \F(G)$ such that $x_M = y_M$.
\item $\|\Phi\| := \sum_{M \in \F_e(G)} \sup_{x \in X} |\Phi(M,x)| < \infty$.
\end{enumerate}

Following \cite{1-gurevich}, we will also assume that potentials are always $G$-invariant, i.e.,
\begin{equation}
\Phi(Mg,x) = \Phi(M,g \cdot x) \text{ for all $x \in X$, $M \in \F(G)$, and $g \in G$.}
\end{equation}

In addition, we define the function $\varphi: X \to \R$ as
\begin{equation}
\varphi(x) := -\sum_{M \in \F_e(G)} |M|^{-1}\Phi(M,x),
\end{equation}
which, since $\|\Phi\| < \infty$, corresponds to a continuous function in $\mathcal{C}(X)$. In \cite{1-tempelman}, $-\varphi$ is called the {\bf local energy} function; here, for consistency with past work, we prefer to include a minus sign `$-$'. For $x \in X$ and $T \in \F(G)$, define the {\bf energy of $x_T$} as
\begin{equation}
E(x_T) := \sum_{M \in \F(T)} \Phi(M,x).
\end{equation}

Similarly, given $y \in X$, we write $x_Ty_{T^{\rm c}} \in S^G$ to denote the {\bf concatenation} of $x_T$ and $y_{T^{\rm c}}$ (i.e., $(x_Ty_{T^{\rm c}})_T = x_T$ and $(x_Ty_{T^{\rm c}})_{T^{\rm c}} = y_{T^{\rm c}}$), and define the {\bf energy of $x_T$ given $y_{T^{\rm c}}$} as
\begin{equation}
E(x_T \vert y_{T^{\rm c}}) := 
\begin{cases}
\sum_{M \in \F_T(G)} \Phi(M,x_Ty_{T^{\rm c}})	&	\text{ if } x_Ty_{T^{\rm c}} \in X,	\\
+\infty							&	\text{ otherwise.}
\end{cases}
\end{equation}

Hence, we can define the {\bf partition function on $T$ with free boundary condition} as
\begin{equation}
Z(T) := \sum_{x_T \in X_T} \exp[-E(x_T)]
\end{equation}
and, analogously, the {\bf partition function on $T$ with boundary condition $y_{T^{\rm c}}$} as
\begin{equation}
Z(T \vert y_{T^{\rm c}}) := \sum_{x_T \in X_T} \exp[-E(x_T \vert y_{T^{\rm c}})],
\end{equation}
where $\exp[-\infty] = 0$. Considering this, the {\bf Gibbs $(X,\Phi)$-specification} $\pi$ is defined as the collection $\pi = \{\pi_T^y\}_{y \in X, T \in \F(G)}$, where
\begin{equation}
\pi^y_T([x_T]) := \frac{\exp[-E(x_T \vert y_{T^{\rm c}})]}{Z(T \vert y_{T^{\rm c}})}.
\end{equation}

A measure $\mu \in \invMeasG$ is a {\bf ($G$-invariant) Gibbs measure} if
\begin{equation}
\mu([x_T] \vert \mathcal{B}_{T^{\rm c}})(y) \stackrel{\mu(y)\text{-a.e.}}{=} \pi^y_T([x_T])
\end{equation}
for all $T \in \F(G)$ and $x_T \in X_T$, where $\mu([x_T] \vert \mathcal{B}_{T^{\rm c}}) := \E_\mu[1_{[x_T]} \vert \mathcal{B}_{T^{\rm c}}]$ is the conditional expectation with respect to $\mathcal{B}_{T^{\rm c}}$ of $1_{[x_T]}$, the characteristic function of $[x_T] \in \B_G$. We denote by $\gibMeas$ the set of Gibbs measures on $X$ for the potential $\Phi$. The set $\gibMeas$ could be empty (e.g., consider \cite[Example (4.16)]{1-georgii}, where it is basically proven that $X = \{\{0,1\}^G: |\{g \in G: x(g) = 1\}| \leq 1\}$ for any denumerable group $G$ cannot support a Gibbs measure), but sometimes it can be guaranteed it is not. For example, if $X$ is an SFT (see \cite{1-ruelle}), then $\mathcal{G}(X,\Phi) \neq \emptyset$.

Finally, given a F{\o}lner sequence $\{T_n\}$ for $G$, we write
\begin{equation}
\mathrm{P}(\Phi,\{T_n\}) := \lim_n |T_n|^{-1} \log Z(T_n)
\end{equation}
whenever the limit exists. If $\mathrm{P}(\Phi,\{T_n\})$ exists and does not depend on $\{T_n\}$, we denote it  by $\mathrm{P}(\Phi)$ and call it the {\bf pressure (of $\Phi$)}.

\subsection{Condition (D) and some implications}

Given $\{T_n\}$ F{\o}lner for $G$, we say that the pair $(X,\{T_n\})$ satisfies {\bf condition (D)} (see \cite[Chapter 4.1]{1-ruelle} for the case $G = \Z^d$) if for all $n \in \N$, there exists $\hat{T}_n \supseteq T_n$ such that $\lim_n \frac{|\hat{T}_n|}{|T_n|} = 1$ and for all $x,y \in X$, there exists $z \in X$ with $z_{T_n} = x_{T_n}$ and $z_{\hat{T}^{\rm c}_n} = y_{\hat{T}^{\rm c}_n}$.

\begin{remark}
Notice that any subsequence of a F{\o}lner sequence is also F{\o}lner. Similarly, condition (D) is also preserved under subsequences.
\end{remark}

From now on, we fix $G$ to be a countable discrete amenable group and $X$ a nonempty $G$-subshift.

\begin{lemma}[{\cite[Lemma 2.3]{2-gurevich}}]
\label{bounds}
If $\{T_n\}$ is a F{\o}lner sequence for $G$ and $(X,\{T_n\})$ satisfies condition (D), then there exists a sequence of positive numbers $\{r_n\}$ with $r_n = o(|T_n|)$ such that for all $\mu \in \mathcal{G}(X,\Phi)$ and $x, \tilde{x} \in X$,
\begin{equation}
\exp[-r_n] \leq \frac{\mu([x_{T_n}]) \exp[E(x_{T_n})]}{\mu([\tilde{x}_{T_n}]) \exp[E(\tilde{x}_{T_n})]} \leq \exp[r_n],
\end{equation}
with $r_n = |\hat{T}_n \setminus T_n|(\log |S| + 4\|\Phi\|) + 2\sup_{z \in X} |E(z_{T_n} \vert z_{T_n^{\rm c}}) - E(z_{T_n})|$. In particular,
\begin{equation}
\label{eq10}
\exp[-r_n] \leq \mu([x_{T_n}]) \exp[E(x_{T_n})] Z(T_n) \leq \exp[r_n],
\end{equation}
uniformly over $x \in X$.
\end{lemma}

\begin{remark}
In \cite[Lemma 2.3]{2-gurevich} it is stated that for every $\mu \in \mathcal{G}(X,\Phi)$ such sequence $\{r_n\}$ exists. However, it can be checked that the sequence does not really depend on any particular choice of $\mu$, and we can pick $\{r_n\}$ so it works uniformly over $\mu \in \mathcal{G}(X,\Phi)$.
\end{remark}

Notice that $|E(x_{T_n})| \leq |T_n|\|\Phi\|$, and therefore $Z(T_n) \leq |S|^{|T_n|}\exp[|T_n|\|\Phi\|]$. In particular, if $\mathrm{P}(\Phi)$ exists, then $\mathrm{P}(\Phi) \leq \log|S| + \|\Phi\|$ and for every $x \in X$,
\begin{align}
\label{smbratio}
0	\leq	-|T_n|^{-1} \log\mu([x_{T_n}])	&	\leq	|T_n|^{-1}\left[E(x_{T_n}) + \log Z(T_n) + o(|T_n|)\right]	\\
								&	\leq	\log|S| + 2\|\Phi\| + o(1),
\end{align}
so $-|T_n|^{-1} \log\mu([x_{T_n}])$ is bounded as a function of $x$ and $n$.

\begin{corollary}
\label{corsupp}
If $\{T_n\}$ is a F{\o}lner sequence for $G$ and $(X,\{T_n\})$ satisfies condition (D), then $\supp(\mu) = X$ for every $\mu \in \mathcal{G}(X,\Phi)$.
\end{corollary}

\begin{proof}
It suffices to prove that $\mu([x_M]) > 0$ for arbitrary $x \in X$ and $M \in \F(G)$. Pick $n \in \N$ sufficiently large so that $Mg \subseteq T_n$ for some $g \in G$. This can be always done (we will prove a stronger statement in Lemma \ref{lemBR}). W.l.o.g., by $G$-invariance of $\mu$, assume that $g = e$. Then, by Lemma \ref{bounds} and Eq. (\ref{eq10}), we conclude that
\begin{equation}
\mu([x_M]) \geq \mu([x_{T_n}]) \geq e^{-r_n} \frac{\exp[-E(x_{T_n})]}{Z(T_n)} \geq  e^{-r_n} |S|^{-|T_n|} e^{-2|T_n| \|\Phi\|} > 0.
\end{equation}
\end{proof}

\section{Ergodic and Shannon-McMillan-Breiman theorems}
\label{sec3}

One of the main goals of this section is to state the ergodic and Shannon-McMillan-Breiman (SMB) theorems, where for the latter it will be required to introduce a special function --the \emph{information function}-- that depends on an algebraic past (and therefore an order) in the group involved. Ultimately, we will show how these two fundamental theorems relate in order to obtain a pressure representation theorem (see Section \ref{sec4}). In particular, by the end of this section, we will discuss a special version of the SMB theorem (the pointwise --or ``Breiman''-- version of it) that holds for Gibbs measures. Before all this, we will need some basic facts concerning decomposition and manipulation of F{\o}lner sequences that will be useful for understanding pressure in the amenable context.

In addition to $G$ and $X$, we will fix $(H,\leqslant)$ to be an amenable linearly right-ordered subgroup $H \leq G$ of finite index $d := [G:H] \in \N$ and $K = \{k_1,\dots,k_d\} \subseteq G$ a left transversal so that $G = KH$ and $k_iH \cap k_jH = \emptyset$ for $i \neq j$.

\subsection{Preliminary facts about F{\o}lner sequences}

Recall that virtually amenable groups are amenable. One way to see this in our context is through the following lemma.

\begin{lemma}
\label{fol1}
If $\{F_n\}$ is F{\o}lner for $H$, then $\{T_n\} = \{KF_n\}$ is F{\o}lner for $G$.
\end{lemma}

\begin{proof}
Given $g \in G$, it can be checked that $gK = \{k_1h_{g,1},\dots,k_dh_{g,d}\}$ for $h_{g,i} \in H$ and $1 \leq i \leq d$. Therefore,
\begin{equation}
\left|gKF_n \triangle KF_n\right| = \left|\left(\bigcup_{i=1}^d k_ih_{g,i}F_n\right) \triangle \left(\bigcup_{i=1}^d k_iF_n\right)\right| \leq \left|\bigcup_{i=1}^d \left(k_ih_{g,i}F_n \triangle k_iF_n\right)\right|.
\end{equation}

To finish, since $|KF_n| = d|F_n|$,
\begin{equation}
\frac{|gKF_n \triangle KF_n|}{|KF_n|} \leq (d|F_n|)^{-1}\left|\bigcup_{i=1}^d k_i (h_{g,i}F_n \triangle F_n)\right| \leq \max_{1 \leq i \leq d} \frac{|h_{g,i}F_n \triangle F_n|}{|F_n|} \xrightarrow[n \to \infty]{} 0,
\end{equation}
because $\{F_n\}$ is F{\o}lner for $H$ and $\{h_{g,i}\}$ is independent of $n$.
\end{proof}

\begin{lemma}
\label{lemQuot}
If $\{T_n\}$ is F{\o}lner for $G$ and $\{T'_n\}$ is such that $T'_n \subseteq T_n$ and $\lim_n |T'_n|/|T_n| = 1$, then $\{T'_n\}$ is F{\o}lner for $G$.
\end{lemma}

\begin{proof}
If $\lim_n |T'_n|/|T_n| = 1$, it can be checked that $\lim_n \frac{|T'_n \triangle T_n|}{|T_n|} = 0$. Then, given $g \in G$,
\begin{align}
\frac{|gT'_n \triangle T'_n|}{|T'_n|}	&	\leq	\frac{|gT'_n \triangle gT_n|}{|T'_n|} + \frac{|gT_n \triangle T_n|}{|T'_n|} + \frac{|T_n \triangle T'_n|}{|T'_n|}	\\
							&	=	\frac{|g(T'_n \triangle T_n)|}{|T_n|}\frac{|T_n|}{|T'_n|} + \frac{|gT_n \triangle T_n|}{|T_n|}\frac{|T_n|}{|T'_n|} + \frac{|T_n \triangle T'_n|}{|T_n|}\frac{|T_n|}{|T'_n|}		\xrightarrow[n \to \infty]{} 0,
\end{align}
and we conclude.
\end{proof}

\begin{lemma}
\label{fol2}
Given $\{T_n\}$ F{\o}lner for $G$, define $F_n := \{h \in H: Kh \subseteq T_n\}$. Then, $KF_n \subseteq T_n$, $\lim_n |KF_n|/|T_n| = 1$ (in particular, $\{KF_n\}$ is F{\o}lner for $G$), and $\{F_n\}$ is F{\o}lner for $H$. Moreover, if $(\{T_n\},X)$ satisfies condition (D), then $(\{KF_n\},X)$ satisfies condition (D), too.
\end{lemma}

\begin{proof}
Let $F_n = \{h \in H: Kh \subseteq T_n\}$ and $Q_n = \{h \in H: Kh \cap T_n \neq \emptyset\}$. Clearly, $KF_n \subseteq T_n \subseteq KQ_n$. Notice that
\begin{equation}
Q_n \setminus F_n = \bigcup_{1 \leq i,j \leq d} \{h \in H: h \in k_i^{-1}T_n, h \notin k_j^{-1}T_n\} = \bigcup_{1 \leq i,j \leq d} (k_i^{-1}T_n \setminus k_j^{-1}T_n) \cap H.			
\end{equation}

As a consequence,
\begin{equation}
T_n \setminus KF_n \subseteq KQ_n \setminus KF_n = K(Q_n \setminus F_n) = \bigcup_{1 \leq i,j \leq d} K(k_i^{-1}T_n \setminus k_j^{-1}T_n).
\end{equation}

Therefore,
\begin{equation}
|T_n \setminus KF_n| \leq \sum_{1 \leq i,j \leq d} |K(k_i^{-1}T_n \setminus k_j^{-1}T_n)| \leq |K|\sum_{1 \leq i,j \leq d} |k_jk_i^{-1}T_n \triangle T_n|,
\end{equation}
so
\begin{align}
\frac{|T_n \setminus KF_n|}{|T_n|} \leq |K|\sum_{1 \leq i,j \leq d} \frac{|k_jk_i^{-1}T_n \triangle T_n|}{|T_n|} \leq d^3 \max_{1 \leq i,j \leq d} \frac{|k_jk_i^{-1}T_n \triangle T_n|}{|T_n|} \xrightarrow[n \to \infty]{} 0,
\end{align}
and $\{KF_n\}$ is F{\o}lner for $G$ due to Lemma \ref{lemQuot}. Now, given $h \in H$, we can see that
\begin{equation}
|hF_n \setminus F_n| = |K|^{-1} |KhF_n \setminus KF_n| \leq |K|^{-1} |Khk_1^{-1}KF_n \setminus KF_n| \leq \max_{g \in Khk_1^{-1}} |gKF_n \setminus KF_n|,
\end{equation}
so
\begin{equation}
\frac{|hF_n \triangle F_n|}{|F_n|} = \frac{|hF_n \setminus F_n|}{|F_n|} + \frac{|h^{-1}F_n \setminus F_n|}{|F_n|} \leq 2\max_{g \in K\{h,h^{-1}\}k_1^{-1}} \frac{|gKF_n \triangle KF_n|}{|KF_n|} \xrightarrow[n \to \infty]{} 0,
\end{equation}
and we conclude that $\{F_n\}$ is F{\o}lner for $H$.

Finally, if $(\{T_n\},X)$ satisfies condition (D), then there exists $\{\hat{T}_n\}$ such that $T_n \subseteq \hat{T}_n$, $\lim_n \frac{|\hat{T}_n|}{|T_n|} = 1$ and for all $x,y \in X$, there exists $z \in X$ such that $z_{T_n} = x_{T_n}$ and $z_{\hat{T}^{\rm c}_n} = y_{\hat{T}^{\rm c}_n}$. Therefore, we can use the same sequence $\{\hat{T}_n\}$ to prove that $(\{KF_n\},X)$ satisfies condition (D), since $KF_n \subseteq T_n$, so $z_{T_n} = x_{T_n}$ implies that $z_{KF_n} = x_{KF_n}$ and  $\frac{|\hat{T}_n|}{|KF_n|} = \frac{|\hat{T}_n|}{|T_n|}\frac{|T_n|}{|KF_n|} \xrightarrow[n \to \infty]{} 1$.
\end{proof}

Thus, in virtue of Lemma \ref{fol1} and Lemma \ref{fol2}, from now on and w.l.o.g., we will assume that we have a sequence $\{T_n\}$ F{\o}lner for $G$ such that $T_n = KF_n$, where $\{F_n\}$ is some F{\o}lner sequence for $H$.

\subsection{Information and entropy}

A {\bf (finite) measurable partition} is a finite collection of disjoint nonempty sets $\{A_1,\dots,A_k\}$ such that $A_i \in \B_G$ and $\bigcup_{i=1}^k A_i = X$ for $k \in \N$. We will only consider the canonical partition of the origin $\alpha = \{[x_e]: x \in X\}$ and its \emph{common refinements} $\alpha^M = \bigvee_{g \in M} g^{-1} \cdot \alpha$, where $M \in \F(G)$ and $g^{-1} \cdot \alpha = \{[x_g]: x \in X\}$. Notice that $\alpha^M = \{[x_M]: x \in X\}$ (after discarding the empty sets) and that $\alpha$ is a \emph{generating partition}, i.e., $\alpha^G = \B_G$; here, especially if $T \subseteq G$ is infinite, we will understand that $\alpha^T$ refers to the smallest $\sigma$-algebra that contains $\sigma(g^{-1} \cdot \alpha)$ for all $g \in T$, i.e., $\B_T$, and there won't be any ambiguity.

Given $\nu \in \invMeasG$, $M \in \F(G)$, and a sub-$\sigma$-algebra $\C$ of $\B_G$, define $\nu(x)$-a.e. the {\bf information function} of $\alpha^M$ conditioned on $\C$ as
\begin{equation}
I_\nu(\alpha^M \vert \C)(x) := -\sum_{A \in \alpha^M} 1_A(x) \log \nu(A \vert \C)(x) = -\log \nu([x_M] \vert \C)(x).
\end{equation}

The {\bf Shannon entropy} of $\alpha^M$ conditioned on $\C$ is defined as
\begin{equation}
H_\nu(\alpha^M \vert \C) := \int{I_\nu(\alpha^M \vert \C)}d\nu.
\end{equation}

We write $I_\nu(\alpha^M)$ and $H_\nu(\alpha^M)$ if $\C$ is the \emph{trivial} sub-$\sigma$-algebra $\{\emptyset,X\}$. Given $T \subseteq G$, we abbreviate $I_\nu(\alpha^M \vert \alpha^T)(x)$ by $I_\nu(M \vert T)(x)$, which is a $\B_{M \cup T}$-measurable function. Notice that if $T \in \F(G)$ and $x \in \supp(\nu,M \cup T)$, we can write $I_\nu(M \vert T)(x) = -\log \nu([x_M] \vert [x_T])$. Given a F{\o}lner sequence $\{T_n\}$, define the {\bf Kolmogorov-Sinai entropy} as the limit
\begin{equation}
h(\nu) := \lim_n |T_n|^{-1} H_\nu(\alpha^{T_n}) = \inf_n |T_n|^{-1} H_\nu(\alpha^{T_n}),
\end{equation}
which turns out to be independent of $\{T_n\}$. The following theorem is a particular case of the \emph{Pinsker Formula} that appears in \cite{1-huang}.

\begin{theorem}[{\cite[Theorem 3.1]{1-huang}}]
\label{pinsker}
Let $G$ be a countable discrete amenable orderable group with algebraic past $G^-$, $X$ a $G$-subshift, and $\nu \in \invMeasG$. Then
\begin{equation}
h(\nu) = H_\nu(\alpha \vert \alpha^{G^-}) = \int{I_\nu(e \vert G^-)}d\nu.
\end{equation}
\end{theorem}

\subsection{Pointwise and mean theorems}

Recall that $ \invAlg_H$ corresponds to the $\sigma$-algebra of $H$-invariant sets. Now we state the ergodic and SMB theorems in their pointwise and mean versions for the particular case of $G$-subshifts.

\begin{theorem}[{\cite{1-lindenstrauss}}]
\label{ergodicThm}
Suppose that $\nu \in \invMeasH$ and $\{F_n\}$ is a $F{\o}lner$ sequence for $H$. Then, for any $f \in L^1_\nu$,
\begin{equation}  
|F_n|^{-1} \sum_{h \in F_n} f(h \cdot x) \xrightarrow[n \to \infty]{L^1_\nu} \E_\nu[f \vert \invAlg_H](x),
\end{equation}
and the same holds $\nu(x)$-a.e. if $\{F_n\}$ is tempered. In particular, if $\nu \in \ergMeasH$, then the limit is constant $\nu(x)$-a.e. and equal to $\int{f}d\nu$.
\end{theorem}

\begin{theorem}[{\cite{1-huang,1-lindenstrauss}}]
Suppose that $\nu \in \invMeasG$ and $\{T_n\}$ is a F{\o}lner sequence for $G$. Then
\begin{equation}
-|T_n|^{-1} \log\nu([x_{T_n}]) \xrightarrow[n \to \infty]{L^1_\nu} \E_\nu[I_{\nu}(e \vert G^-) \vert \invAlg_G](x),
\end{equation}
and the same holds $\nu(x)$-a.e. if $\{T_n\}$ is tempered. In particular, if $\nu \in \ergMeasG$, then the limit is constant $\nu(x)$-a.e. and equal to $h(\nu)$ (see Theorem \ref{pinsker}).
\end{theorem}

We will refer to $-|T_n|^{-1} \log\nu([x_{T_n}])$ as \emph{SMB ratio}. Notice that in the Gibbsian case we have already established that this ratio is uniformly bounded over $x \in X$ and $n \in \N$ (see Eq. (\ref{smbratio})). Now we will review a couple of results related to a Breiman type theorem due to Gurevich and Tempelman \cite{1-gurevich}.

\subsection{Breiman type theorem for Gibbs measures}

Given a F{\o}lner sequence $\{T_n\}$ for $G$, define the set
\begin{equation}
X_{\varphi,\{T_n\}} := \left\{x \in X: \lim_n |T_n|^{-1} \sum_{g \in T_n} \varphi(g \cdot x) \text{ exists}\right\}.
\end{equation}

\begin{theorem}[{\cite[Theorem 1]{1-gurevich}}]
\label{thm1}
If $\{T_n\}$ is F{\o}lner for $G$, $(X,\{T_n\})$ satisfies condition (D), and $\Phi$ is a potential such that $\mathcal{G}(X,\Phi) \neq \emptyset$, then
\begin{enumerate}
\item $\mathrm{P}(\Phi)$ is well defined,
\item $\mathrm{P}(\Phi) = \sup_{\nu \in \invMeasG}\left[h(\nu) + \int{\varphi}d\nu\right] = h(\mu) + \int{\varphi}d\mu$ for any $\mu \in \mathcal{G}(X,\Phi)$, and
\item for $x \in X_{\varphi,\{T_n\}}$ and $\mu \in \mathcal{G}(X,\Phi)$,
\begin{equation}
\lim_n \left[-|T_n|^{-1} \log \mu([x_{T_n}]) \right] = -\lim_n |T_n|^{-1} \sum_{g \in T_n} \varphi(g \cdot x) + \mathrm{P}(\Phi).
\end{equation}
\end{enumerate}
\end{theorem}

Part 2 of Theorem \ref{thm1} can be regarded as a \emph{variational principle} (see, for example, \cite{2-walters,1-misiurewicz}). Then, combining Theorem \ref{pinsker} and Theorem \ref{thm1}, we have that
\begin{equation}
\mathrm{P}(\Phi) = \int{(I_\mu(e \vert G^-) + \varphi)}d\mu
\end{equation}
for all $\mu \in \mathcal{G}(X,\Phi)$. The main goal of this paper is to give conditions in order to replace $d\mu$ by $d\nu$ in the previous expression, where $\nu$ is some other measure in $\invMeasH$ and the integrand is the same or a related functional. We also have the following result.

\begin{lemma}[{\cite[Lemma 2]{1-gurevich}}]
\label{lem2}
If $\{T_n\}$ is a F{\o}lner sequence for $G$, then
\begin{eqnarray}
\lim_n |T_n|^{-1} \left[- \sum_{g \in T_n} \varphi(g \cdot x) - E(x_{T_n})\right] = 0	&	\text{uniformly over $x \in X$.}
\end{eqnarray}
\end{lemma}

Combining Lemma \ref{lem2}, Eq. (\ref{eq10}), and the definition of $T_n$, we obtain
\begin{eqnarray}
&		&	-|T_n|^{-1} \log \mu([x_{T_n}])											\\
&	=	&	|T_n|^{-1}(\log Z(T_n) - \sum_{g \in T_n} \varphi(g \cdot x) + o(|T_n|))				\\
&	=	&	|KF_n|^{-1}\log Z(KF_n) - |K|^{-1} \sum_{i=1}^{[G:H]} |F_n|^{-1} \sum_{h \in F_n} \varphi_i(h \cdot x) + o(1),
\end{eqnarray}
where $\varphi_i(x) := \varphi(k_i \cdot x)$. Denote $\varphi_K(x) := \sum_{i=1}^{[G:H]} \varphi_i(x)$. Then, applying the pointwise ergodic theorem (see Theorem \ref{ergodicThm}) for each $\varphi_i$, and considering that $|K| = [G:H]$, we have the following result.

\begin{corollary}[{\cite[Corollary 2]{1-gurevich}}]
\label{cor1}
If $\{F_n\}$ is a tempered F{\o}lner sequence for $H$, $(X,\{T_n\})$ satisfies condition (D), and $\Phi$ is a potential such that $\mathcal{G}(X,\Phi) \neq \emptyset$, then
\begin{equation}
-|T_n|^{-1} \log \mu([x_{T_n}]) \xrightarrow[n \to \infty]{\nu(x)\text{-a.e.}} \mathrm{P}(\Phi) - [G:H]^{-1} \E_\nu\left[\varphi_K \vert \invAlg_H\right](x)
\end{equation}
for any $\nu \in \invMeasH$.
\end{corollary}

In addition, if we apply the bounded convergence theorem (BCT), and given the just established pointwise convergence, we also have the following corollary.

\begin{corollary}
\label{cor2}
If $\{F_n\}$ is a tempered F{\o}lner sequence for $H$, $(X,\{T_n\})$ satisfies condition (D), and $\Phi$ is a potential such that $\mathcal{G}(X,\Phi) \neq \emptyset$, then
\begin{equation}
-|T_n|^{-1} \log \mu([x_{T_n}])	\xrightarrow[n \to \infty]{L^1_\nu} \mathrm{P}(\Phi) - [G:H]^{-1} \E_\nu\left[\varphi_K \vert \invAlg_H\right](x)
\end{equation}
for any $\nu \in \invMeasH$.
\end{corollary}

\begin{remark}
A version of Corollary \ref{cor2} (i.e., $L^1_\nu$ convergence of the SMB ratio) without requiring $\{T_n\}$ to be tempered nor going through pointwise convergence can be also proven. However, since all the conditions on F{\o}lner sequences of the main pressure representation theorem (see Theorem \ref{pressureRep} and Section \ref{sec4} in general) remain true after taking subsequences and any F{\o}lner sequence has a tempered subsequence, this is enough.
\end{remark}

\section{Pressure representation theorem}
\label{sec4}

In this section we will state and prove the main result of this work, namely Theorem \ref{pressureRep}. In order to achieve this, its proof has been divided in several steps, each of them made up of lemmas and propositions. There can be distinguished three main steps in the proof, each of them encompassed in a corresponding subsection. First, given a F{\o}lner sequence, we will show that the SMB ratio can be expressed as a summation of several information functions conditioned on finite portions of cosets from the algebraic past of an orderable subgroup of $G$ (which itself is not necessarily orderable). This decomposition is a more general and sophisticated version of what Gamarnik and Katz called the \emph{sequential cavity method} in \cite{1-gamarnik} and it is also related to ideas from Helvik and Lindgren \cite{1-helvik} (see Section \ref{sec5} for a discussion on this last point). Second, we group all these information functions in finitely many specific classes, and prove that when ``most'' of the elements in each class are ``similar'' to a particular function, then we can represent pressure using a formula involving finitely many representative functions, one for each class. Third, we define a natural way to group these information functions in classes which at the same time induces a notion of convergence in a net sense to the corresponding representative function. Finally, we put all these elements together and state Theorem \ref{pressureRep}.

Recall that we have fixed $G$ to be a countable amenable virtually orderable group with $(H,\leqslant)$ an amenable linearly right-ordered subgroup $H \leq G$ of finite index $[G:H]$ and $K = \{k_1,\dots,k_{[G:H]}\} \subseteq G$ a left transversal so that $G = KH$ and $k_iH \cap k_{i'}H = \emptyset$ for $i \neq i'$. In addition, we will fix $\{F_n\}$ a F{\o}lner sequence for $H$, so that $\{T_n\}$ is F{\o}lner for $G$, where $T_n = KF_n$ and $(X,\{T_n\})$ satisfies condition (D). Finally, we will also consider an arbitrary partition of $K$ into $\ell$ nonempty subsets $\{L_j\}$ so that $K = L_1 \sqcup \cdots \sqcup L_\ell$ and $L_{j} \cap L_{j'} = \emptyset$ for $j \neq j'$, and we will denote $K_i := L_1 \sqcup \cdots \sqcup L_i$ for $1 \leq i \leq \ell$ and $K_0 = \emptyset$. In particular, $K_\ell = K$.

\subsection{Sequential decomposition of the SMB ratio}

Now, we will proceed to decompose the SMB ratio into a particular sum of conditional probabilities. Let $X$ be a nonempty $G$-subshift and $\mu \in \invMeasG$. Notice that

\begin{equation}
\mu([x_{T_n}]) = \prod_{i=1}^{\ell} \mu([x_{L_iF_n}] \vert [x_{K_{i-1}F_n}]).
\end{equation}

Given $F \in \F(H)$, since $H$ is linearly right-ordered, we can always define a unique maximal element $h \in F$, so that $F \setminus \{h\} = F \cap H^-h$. Iterating this idea on each set $L_iF_n$, and due to the $G$-invariance (in particular, $H$-invariance) of $\mu$, we have
\begin{eqnarray}
&		&	\mu([x_{L_iF_n}] \vert [x_{K_{i-1}F_n}])																				\\
&	=	&	\prod_{h \in F_n} \mu([x_{L_ih }] \vert [x_{L_i(F_n \cap H^-h) \sqcup K_{i-1}F_n}])										\\
&	=	&	\prod_{h \in F_n} \mu(h^{-1} \cdot [(h \cdot x)_{L_i}] \vert h^{-1} \cdot [(h \cdot x)_{L_i(F_nh^{-1} \cap H^-) \sqcup K_{i-1}F_nh^{-1}}])	\\
&	=	&	\prod_{h \in F_n} \mu([(h \cdot x)_{L_i}] \vert [(h \cdot x)_{L_i(F_nh^{-1} \cap H^-) \sqcup K_{i-1}F_nh^{-1}}])					\\
&	=	&	\prod_{h \in F_n} \mu([(h \cdot x)_{L_i}] \vert [(h \cdot x)_{K_iF_nh^{-1} \cap G^-_i}])
\end{eqnarray}
for $1 \leq i \leq \ell$, where
\begin{equation}
G^-_i := L_iH^- \sqcup K_{i-1}H.
\end{equation}

Notice that $G^-_{i-1} \subsetneq K_{i-1}H \subsetneq G^-_{i}$. In addition, given $h \in T_n$, define the set
\begin{equation}
T^-_{n,h}(i) := K_iF_nh^{-1} \cap G^-_i = KF_nh^{-1} \cap G^-_i = T_nh^{-1} \cap G^-_i.
\end{equation}

Then, for $1 \leq i \leq \ell$, we can write
\begin{equation}
\label{identity}
-\log\mu([x_{K_iF_n}]) = \sum_{j=1}^{i} \sum_{h \in F_n} I_\mu(L_i \vert T^-_{n,h}(i))(h \cdot x),
\end{equation}
where $ I_\mu(L_i \vert T^-_{n,h}(i))(x) := -\log\mu\left([x_{L_i}] \middle\vert [x_{T^-_{n,h}(i)}]\right)$ is the information function of $\alpha^{L_i}$ conditioned on $\alpha^{T^-_{n,h}(i)} = \B_{T^-_{n,h}(i)}$. In particular,
\begin{equation}
\label{formula}
-\log\mu([x_{T_n}]) = \sum_{i=1}^{\ell} \sum_{h \in F_n} I_\mu(L_i \vert T^-_{n,h}(i))(h \cdot x).
\end{equation}

Whenever $\ell = 1$, we will omit the index $i$ in expressions like $I_\mu(L_i \vert T^-_{n,h}(i))(x)$, i.e., we will write $I_\mu(K \vert T^-_{n,h})(x)$ instead of $I_\mu(L_1 \vert T^-_{n,h}(1))(x)$, etc.

The cases $\ell = 1$ and $\ell = [G:H]$ represent two extreme cases in the way of decomposing $T_n$, both --together with the intermediate cases-- potentially useful. When $\ell = 1$, we are splitting $T_n = KF_n$ in disjoint $|F_n|$ translations of $K$, let's say $\{Kh\}_{h \in F_n}$, and ordering them according to the order of $h$ in $H$. On the other hand, when $\ell = [G:H]$, what we do is to decompose $T_n$ in $[G:H]$ ``slices'' $k_1F_n,\dots,k_{\ell}F_n$ arbitrarily ordered and order each of them according to the order of $H$, giving two alternative ways of expressing the same quantity. We will see that this decomposition generalizes what is done in \cite[Theorem 1]{1-helvik} for the case of finite unions of \emph{point lattices} (i.e. discrete abelian subgroups of $\R^d$), where they also include some applications (see Section \ref{mueqnu}).

\subsection{Convergence results}

We are interested in the case when ``most'' of the elements $I_\mu(L_i \vert T^-_{n,h}(i))$ in the previous averages are ``similar'' to a function $f_i \in L^1_\nu$ in some sense to be made precise. This idea was present along all papers in the sequence \cite{1-gamarnik,1-marcus,1-briceno,1-adams} and it will be now formalized in a mean sense here. Consider the two following propositions.

\begin{proposition}
\label{propIFF}
Let $\{F_n\}$ be a F{\o}lner sequence for $H$, $\mu \in \gibMeas$, $\nu \in \invMeasH$ with $\supp(\nu) \subseteq \supp(\mu)$, and $f_1,\dots,f_\ell \in L^1_\nu$. Then,
\begin{eqnarray}
|F_n|^{-1} \sum_{h \in F_n} (I_\mu(L_i \vert T^-_{n,h}(i)) - f_i)(h \cdot x)  \xrightarrow[n \to \infty]{L^1_\nu} 0	&	\text{for all } 1 \leq i \leq \ell
\end{eqnarray}
if and only if
\begin{eqnarray}
-|F_n|^{-1}\log \mu([x_{K_iF_n}]) \xrightarrow[n \to \infty]{L^1_\nu} \E_\nu[f_1 + \cdots + f_i \vert \invAlg_H]	&	\text{for all } 1 \leq i \leq \ell.
\end{eqnarray}
\end{proposition}

\begin{proof}
By the mean ergodic theorem, for any $1 \leq i \leq \ell$,
\begin{equation}
|F_n|^{-1} \sum_{h \in F_n} (I_\mu(L_i \vert T^-_{n,h}(i)) - f_i)(h \cdot x)  \xrightarrow[n \to \infty]{L^1_\nu} 0
\end{equation}
if and only if
\begin{equation}
\label{equiv}
|F_n|^{-1} \sum_{h \in F_n} I_\mu(L_i \vert T^-_{n,h}(i))(h \cdot x)  \xrightarrow[n \to \infty]{L^1_\nu} \E_\nu[f_i \vert \invAlg_H].
\end{equation}

Then, by Eq. (\ref{identity}), Eq. (\ref{equiv}), and linearity of conditional expectation,
\begin{eqnarray}
-|F_n|^{-1}\log \mu([x_{K_iF_n}])	&	= 							&	\sum_{j=1}^{i}  |F_n|^{-1} \sum_{h \in F_n} I_\mu(L_i \vert T^-_{n,h}(j))(h \cdot x)	\\
							&	\xrightarrow[n \to \infty]{L^1_\nu}	&	\sum_{j=1}^{i}\E_\nu[f_j \vert \invAlg_H](x).
\end{eqnarray}

Now, to prove the other direction, notice that
\begin{eqnarray}
&								&	|F_n|^{-1} \sum_{h \in F_n} I_\mu(L_i \vert T^-_{n,h}(i))(h \cdot x)				\\
&	= 							&	|F_n|^{-1}(-\log \mu([x_{K_iF_n}]) + \log \mu([x_{K_{i-1}F_n}]))					\\
&	\xrightarrow[n \to \infty]{L^1_\nu}	&	\E_\nu[f_1+\cdots+f_i \vert \invAlg_H] - \E_\nu[f_1+\cdots+f_{i-1} \vert \invAlg_H]	\\
&	=							&	\E_\nu[f_i \vert \invAlg_H],
\end{eqnarray}
and we conclude by Eq. (\ref{equiv}).
\end{proof}

Combining Proposition \ref{propIFF} and Corollary \ref{cor2}, we obtain the following result, which gives us a way to relate the SMB ratio, the average of information functions, and pressure.

\begin{proposition}
\label{mainProp}
Let $\{F_n\}$ be a F{\o}lner sequence for $H$ such that $(X,\{KF_n\})$ satisfies condition (D), $\Phi$ a potential such that $\mathcal{G}(X,\Phi) \neq \emptyset$, $\nu \in \invMeasH$, and $f_1,\dots,f_\ell \in L^1_\nu$. If
\begin{eqnarray}
\label{hypothesis}
|F_n|^{-1} \sum_{h \in F_n} (I_\mu(L_i \vert T^-_{n,h}(i)) - f_i)(h \cdot x)  \xrightarrow[n \to \infty]{L^1_\nu} 0	&	\text{for all } 1 \leq i \leq \ell,
\end{eqnarray}
then
\begin{eqnarray}
\mathrm{P}(\Phi)	&	=						&	\frac{1}{[G:H]} \int{\left(\sum_{i=1}^\ell f_i + \varphi_K\right)}d\nu	\\
		&	\stackrel{\nu(x)\text{-a.e.}}{=}	&	\frac{1}{[G:H]} \E_\nu\left[\sum_{i=1}^\ell f_i + \varphi_K \middle\vert \invAlg_H\right](x).
\end{eqnarray}

On the other hand, if $\ell = 1$, we also have the following converse:
\begin{eqnarray}
	&	\mathrm{P}(\Phi) \stackrel{\nu(x)\text{-a.e.}}{=} \E_\nu[f + \varphi_K \vert \invAlg_H](x)						\\
\iff	&	|F_n|^{-1} \sum_{h \in F_n} (I_\mu(K \vert T^-_{n,h}) - f)(h \cdot x)  \xrightarrow[n \to \infty]{L^1_\nu} 0.
\end{eqnarray}
\end{proposition}

\begin{proof}
By Corollary \ref{cor2},
\begin{equation}
\label{EqA}
-|T_n|^{-1} \log \mu([x_{T_n}])	\xrightarrow[n \to \infty]{L^1_\nu} \mathrm{P}(\Phi) -  [G:H]^{-1} \E_\nu\left[\varphi_K \vert \invAlg_H\right](x).
\end{equation}

On the other hand, by Proposition \ref{propIFF}, Eq. (\ref{hypothesis}) is equivalent to
\begin{equation}
-|K_iF_n|^{-1}\log \mu([x_{K_iF_n}]) \xrightarrow[n \to \infty]{L^1_\nu} |K_i|^{-1} \E_\nu[f_1 + \cdots + f_i \vert \invAlg_H]
\end{equation}
for all $1 \leq i \leq \ell$. In particular, for $i = \ell$,
\begin{equation}
\label{EqB}
-|KF_n|^{-1}\log \mu([x_{KF_n}]) \xrightarrow[n \to \infty]{L^1_\nu} [G:H]^{-1} \E_\nu[f_1 + \cdots + f_\ell \vert \invAlg_H].
\end{equation}

Therefore, since $KF_n = T_n$ and the limits in Eq. (\ref{EqA}) and Eq. (\ref{EqB}) should coincide $\nu(x)$-a.e., we have
\begin{equation}
\mathrm{P}(\Phi) - [G:H]^{-1} \E_\nu\left[\varphi_K \vert \invAlg_H\right](x) \stackrel{\nu(x)\text{-a.e.}}{=} [G:H]^{-1} \E_\nu[f_1 + \cdots + f_\ell \vert \invAlg_H],
\end{equation}
so
\begin{equation}
\mathrm{P}(\Phi) \stackrel{\nu(x)\text{-a.e.}}{=} \frac{1}{[G:H]} \E_\nu\left[\sum_{i=1}^\ell f_i + \varphi_K \middle\vert \invAlg_H\right](x).
\end{equation}

Since $\mathrm{P}(\Phi)$ is constant, then $\mathrm{P}(\Phi) = \int{\mathrm{P}(\Phi)}d\nu$, and we conclude that
\begin{eqnarray}
\mathrm{P}(\Phi)	&	=						&	\frac{1}{[G:H]} \sum_{i=1}^\ell \int{(f_i + \frac{1}{\ell} \varphi_K)}d\nu	\\
		&	\stackrel{\nu(x)\text{-a.e.}}{=}	&	\frac{1}{[G:H]} \sum_{i=1}^\ell \E_\nu[f_i+ \frac{1}{\ell} \varphi_K \vert \invAlg_H](x).
\end{eqnarray}

Now, if $\ell = 1$, Eq. (\ref{hypothesis}) is equivalent to
\begin{equation}
|F_n|^{-1} \sum_{h \in F_n} (I_\mu(K \vert T^-_{n,h}) - f)(h \cdot x)  \xrightarrow[n \to \infty]{L^1_\nu} 0,
\end{equation}
which at the same time is equivalent to
\begin{equation}
-|KF_n|^{-1}\log \mu([x_{KF_n}]) \xrightarrow[n \to \infty]{L^1_\nu} [G:H]^{-1} \E_\nu[f \vert \invAlg_H].
\end{equation}

Then, applying Corollary \ref{cor2} again, we prove the converse.
\end{proof}

We have obtained a pressure representation formula when $I_\mu(L_i \vert T^-_{n,h}(i))$ is ``similar'' to $f_i \in L^1_\nu$ in the average sense of Eq. (\ref{hypothesis}). In the following, we will see a way to define a natural candidate for such $f_i$. 

\subsection{A natural directed set}

In the following, we will define an order and directed set which ultimately gives us a notion of convergence in a net sense for information functions, and for which the representative functions $f_i$ naturally emerge as limiting functions.

Given $\{F_n\}$ F{\o}lner for $H$, define the directed set $([\{F_n\}], \preccurlyeq)$ as the countable collection
\begin{equation}
[\{F_n\}] = \left\{(n,h): n \in \N, h \in F_n\right\} \subseteq \N \times H
\end{equation}
such that
\begin{equation}
(n_1,h_1) \preccurlyeq (n_2,h_2) \iff n_1 \leq n_2 \text{ and } T_{n_1,h_1}^-(i) \subseteq T_{n_2,h_2}^-(i) \text{ for all } 1 \leq i \leq \ell,
\end{equation}
where $T_{n,h}^-(i) = T_nh^{-1} \cap G^-_i$, $G^-_i = L_iH^- \sqcup K_{i-1}H$, and $K$ is so that $G = KH$.

To see that this is in fact a directed set, we need to prove that there is an upper bound for each pair of elements $(n_1,h_1), (n_2,h_2) \in [\{F_n\}]$. We will use the following lemma.

\begin{lemma}
\label{lemBR}
Let $\{F_n\}$ be a F{\o}lner sequence for $H$. Given a fixed $M \in \F(H)$, define
\begin{equation}
F^M_n := \{h \in F_n: Mh \subseteq T_n\}
\end{equation}
for $n \in \N$. Then $\lim_n \frac{|F^M_n|}{|F_n|} = 1$. Moreover, if $M_1, M_2,\dots$ is a family of sets in $\F(H)$, there exists an increasing function $\gamma: \N \to \N$ such that $\gamma(n) \to \infty$ and for all $n \in \N$,
\begin{equation}
\label{setIneq}
\frac{|F^{M_{\gamma(n)}}_n|}{|F_n|} \geq 1 - \frac{1}{\gamma(n)}.
\end{equation}
\end{lemma}

\begin{proof}
If $h \in F_n$ and $Mh \nsubseteq F_n$, then for some $m \in M$, $h \in F_n \setminus (m^{-1}F_n)$. Thus,
\begin{equation}
\frac{|F_n \setminus F^M_n|}{|F_n|} \leq \frac{\sum_{m \in M} |F_n \setminus (m^{-1}F_n)|}{|F_n|}  \leq |M| \max_{m \in M} \frac{|mF_n \triangle F_n|}{|F_n|} \to 0.
\end{equation}

Now, define $n(1) = 1$. For every $j \geq 2$, there exists a minimum $n(j) > n(j-1)$ so that
\begin{eqnarray}
\frac{|F_n^{M_j}|}{|F_n|} \geq 1 - \frac{1}{j}	&	\text{for all $n \geq n(j)$}.
\end{eqnarray}

We can define $\gamma$ as $\gamma(n) = j$ for $n(j) \leq n < n(j+1)$. Then
\begin{eqnarray}
\frac{|F_n^{M_{\gamma(n)}}|}{|F_n|} \geq 1 - \frac{1}{\gamma(n)}	&	\text{for all $n(j+1) > n \geq n(j)$},
\end{eqnarray}
and considering that $\N = \bigcup_j [n(j),n(j+1))$, we have that Eq. (\ref{setIneq}) holds.
\end{proof}

Then, by Lemma \ref{lemBR}, given $(n_1,h_1), (n_2,h_2) \in [\{F_n\}]$, it suffices to take $(n,h)$ so that $n$ is sufficiently large such that $n \geq \max\{n_1,n_2\}$ and $F_n^M \neq \emptyset$, where $M = F_{n_1}h_1^{-1} \cup F_{n_2}h_2^{-1}$. Then, if we take any $h \in F_n^M$, we have that $Mh \subseteq F_n$, so $F_{n_1}h_1^{-1}, F_{n_2}h_2^{-1} \subseteq F_nh^{-1}$. In particular, for all $1 \leq i \leq \ell$,
\begin{equation}
T^-_{n_1,h_1}(i) = KF_{n_1}h_1^{-1} \cap G^-_i \subseteq KF_nh^{-1} \cap G^-_i = T^-_{n,h}(i),
\end{equation}
and the same holds for $T^-_{n_2,h_2}(i)$.

Notice that $[\{F_n\}]$ has the additional property that if $\{F_{n_k}\}$ is a subsequence of $\{F_n\}$, then $([\{F_{n_k}\}], \preccurlyeq)$ is a directed subset .

Now, suppose that there exists $I_{\mu,\nu}(L_i \vert G^-_i) \in L^1_\nu$ such that $I_\mu(L_i \vert T^-_{n,h}(i))$ converges to $I_{\mu,\nu}(L_i \vert G^-_i)$ in $L^1_\nu$ in the $[\{F_n\}]$-sense, this is to say, for every $\epsilon > 0$, there exists $(n_0,h_0) \in [\{F_n\}]$ such that for every $(n,h) \succcurlyeq (n_0,h_0)$,
\begin{eqnarray}
\|I_\mu(L_i \vert T^-_{n,h}(i)) - I_{\mu,\nu}(L_i \vert G^-_i)\|_{\nu} < \epsilon	&	\text{for all $1 \leq i \leq \ell$.}
\end{eqnarray}

Notice that if this is the case, then $I_{\mu,\nu}(L_i \vert G^-_i)(x) \geq 0$ $\nu(x)$-a.e. since $I_\mu(L_i \vert T^-_{n,h}(i))(x) \geq 0$ for all $x \in X$. We will abbreviate this fact by $I_\mu(L_i \vert T^-_{n,h}(i)) \xrightarrow[{[\{F_n\}]}]{L^1_\nu} I_{\mu,\nu}(L_i \vert G^-_i)$ and we will consider this as a necessary condition for having any meaningful pressure representation as in \cite{1-gamarnik,1-marcus,1-briceno,1-adams}, where analogous convergence assumptions were required to be uniform instead. It is reasonable to relax this to a mean sense; after all, what we regard as a pressure representation involves an integration with respect to $\nu$, so asking limits to be in an $L^1_\nu$ sense is a natural assumption.
 
\begin{remark}
A priori, $I_{\mu,\nu}(L_i \vert G^-_i)$ could depend on the particular F{\o}lner sequence $\{F_n\}$ that we use, but we won't reflect this fact in the notation because it is already heavy enough.
\end{remark}

Now, if we assume $L^1_\nu$ convergence in the $[\{F_n\}]$-sense for each $i$, and after removing $o(|F_n|)$ elements from each $F_n$, we can prove that the $i$th average of information functions arising from the sequential decomposition of the SMB ratio (where each information function is conditioned on a different set) converges in $L^1_\nu$ to the ergodic average of $I_{\mu,\nu}(L_i \vert G^-_i)$ (that we naturally identify with the aforementioned representative function $f_i$). We have the following lemma.

\begin{proposition}
\label{prop2}
If $I_\mu(L_i \vert T^-_{n,h}(i)) \xrightarrow[{[\{F_n\}]}]{L^1_\nu} I_{\mu,\nu}(L_i \vert G^-_i)$ for all $1 \leq i \leq \ell$, then we can define a sequence $\{F_{n,\nu}\}$ such that $F_{n,\nu} \subseteq F_n \subseteq H$, $\lim_n \frac{|F_{n,\nu}|}{|F_n|} = 1$, and
\begin{equation}
|F_n|^{-1} \left[\sum_{h \in F_{n,\nu}} I_\mu(L_i \vert T^-_{n,h}(i))(h \cdot x) - \sum_{h \in F_n} I_{\mu,\nu}(L_i \vert G^-_i)(h \cdot x)\right] \xrightarrow[n \to \infty]{L^1_\nu} 0
\end{equation}
for all $1 \leq i \leq \ell$.
\end{proposition}

\begin{proof}
Since $I_\mu(L_i \vert T^-_{n,h}(i)) \xrightarrow[{[\{F_n\}]}]{L^1_\nu} I_{\mu,\nu}(L_i \vert G^-_i)$ for all $1 \leq i \leq \ell$ and $[\{F_n\}]$ is directed, we can find an increasing sequence $(n_m,h_m) \in [\{F_n\}]$ such that for every $(n,h) \succcurlyeq (n_m,h_m)$,
\begin{eqnarray}
\left\|I_\mu(L_i \vert T^-_{n,h}(i)) - I_{\mu,\nu}(L_i \vert G^-_i)\right\|_{\nu} < \frac{1}{m}	&	\text{for all $1 \leq i \leq \ell$.}
\end{eqnarray}

Now, consider the sequence of finite subsets of $H$ defined as $M_m = F_{n_m}h_m^{-1},$ and let $\gamma$ be the (increasing) function provided by Lemma \ref{lemBR}. Set $F_{n,\nu}$ to be $F_n^{M_{\gamma(n)}}$, so $\frac{|F_{n,\nu}|}{|F_n|} \geq 1-\frac{1}{\gamma(n)}$.

It can be seen that for $n \in \N$ and $h \in F_{n,\nu}$, we have
\begin{eqnarray}
\left\|I_\mu(L_i \vert T^-_{n,h}(i)) - I_{\mu,\nu}(L_i \vert G^-_i)\right\|_{\nu} < \frac{1}{\gamma(n)}	&	\text{for all $1 \leq i \leq \ell$.}
\end{eqnarray}

Therefore, by $H$-invariance of $\nu$, for any $1 \leq i \leq \ell$,
\begin{align}
\left\|\sum_{h \in F_{n,\nu}} \left(I_\mu(L_i \vert T^-_{n,h}(i)) - I_{\mu,\nu}(L_i \vert G^-_i)\right)(h \cdot x)\right\|_{\nu}				\\
\leq \sum_{h \in F_{n,\nu}}\|I_\mu(L_i \vert T^-_{n,h}(i)) - I_{\mu,\nu}(L_i \vert G^-_i)\|_{\nu}	 \leq \frac{|F_{n,\nu}|}{\gamma(n)}.
\end{align}

In addition, since $I_{\mu,\nu}(L_i \vert G^-_i) \in L^1_\nu$ and $I_{\mu,\nu}(L_i \vert G^-_i)(x) \geq 0$ $\nu(x)$-a.e., by $H$-invariance of $\nu$,
\begin{align}
\left\|\sum_{h \in F_n \setminus F_{n,\nu}} I_{\mu,\nu}(L_i \vert G^-_i)(h \cdot x)\right\|_{\nu}	&	= \sum_{h \in F_n \setminus F_{n,\nu}} \left\|I_{\mu,\nu}(L_i \vert G^-_i)(h \cdot x)\right\|_{\nu}	\\
																		&	= |F_n \setminus F_{n,\nu}|\left\|I_{\mu,\nu}(L_i \vert G^-_i)\right\|_{\nu}.
\end{align}

Therefore, we conclude by observing that
\begin{align}
\left\|\sum_{h \in F_{n,\nu}} I_\mu(L_i \vert T^-_{n,h}(i))(h \cdot x) - \sum_{h \in F_n} I_{\mu,\nu}(L_i \vert G^-_i)(h \cdot x)\right\|_{\nu}	\\
\leq \frac{|F_{n,\nu}|}{\gamma(n)} + |F_n \setminus F_{n,\nu}||\left\|I_{\mu,\nu}(L_i \vert G^-_i)\right\|_{\nu},
\end{align}
and using $\|I_{\mu,\nu}(L_i \vert G^-_i)\|_\nu < \infty$, $|F_{n,\nu}| \leq |F_n|$, and $\lim_n \frac{|F_n \setminus F_{n,\nu}|}{|F_n|} = \lim_n \frac{1}{\gamma(n)} = 0$.
\end{proof}

\begin{remark}
Notice that
\begin{eqnarray}
		&	\sup_n \max_{h \in F_n \setminus F_{n,\nu}} \left\|I_\mu(L_i \vert T^-_{n,h}(i))\right\|_\nu < \infty						\\	
\iff		&	\limsup_n \max_{h \in F_n \setminus F_{n,\nu}} \|I_\mu(L_i \vert T^-_{n,h}(i))\|_{\nu} < \infty						\\
\implies	& 	|F_n|^{-1} \sum_{h \in F_n \setminus F_{n,\nu}} \|I_\mu(L_i \vert T^-_{n,h}(i))\|_{\nu}  \xrightarrow[n \to \infty]{} 0			\\
\iff		&	|F_n|^{-1} \sum_{h \in F_n \setminus F_{n,\nu}} I_\mu(L_i \vert T^-_{n,h}(i))(h \cdot x)  \xrightarrow[n \to \infty]{L^1_\nu} 0.
\end{eqnarray}
\end{remark}

\subsection{Main theorem}

Now, combining the previous results from this section, we are in position to state the main theorem of this work.

\begin{theorem}
\label{pressureRep}
Consider
\begin{itemize}
\item $G$ a countable discrete group,
\item $(H,\leqslant)$ an amenable linearly right-ordered subgroup $H \leq G$ of index $[G:H] < \infty$,
\item $K = \{k_1,\dots,k_d\} \subseteq G$ a left transversal so that $d = [G:H]$ and $G = KH$,
\item A partition $K = L_1 \sqcup \cdots \sqcup L_\ell$ with $L_i \neq \emptyset$ and $1 \leq \ell \leq d$,
\item $X$ a nonempty $G$-subshift,
\item $\{F_n\}$ a left F{\o}lner sequence for $H$ such that $(X,\{KF_n\})$ satisfies condition (D),
\item $\Phi$ an (absolutely summable) potential such that $\mathcal{G}(X,\Phi) \neq \emptyset$,
\item $\mu \in \gibMeas \subseteq \invMeasG$, and $\nu \in \invMeasH$.
\end{itemize}

For all $1 \leq i \leq \ell$, denote
\begin{itemize}
\item $K_i = L_1 \sqcup \cdots \sqcup L_i$,
\item $G^-_i = L_iH^- \sqcup K_{i-1}H$, where $H^- = \{h \in H: h < e\}$,
\item $T_{n,h}^-(i) = T_nh^{-1} \cap G^-_i$, where $T_n = KF_n$,
\item $I_\mu(L_i \vert T^-_{n,h}(i))(x) = -\log\mu([x_{L_i}] \vert [x_{T_{n,h}^-(i)}])$, and
\item $\varphi_K(x) = \sum_{i=1}^d \varphi_i(x)$, where $\varphi_i(x) = \varphi(k_i \cdot x)$.
\end{itemize}

Suppose that
\begin{eqnarray}
I_\mu(L_i \vert T^-_{n,h}(i)) \xrightarrow[{[\{F_n\}]}]{L^1_\nu} I_{\mu,\nu}(L_i \vert G^-_i)	&	\text{for all } 1 \leq i \leq \ell.
\end{eqnarray}

If
\begin{eqnarray}
|F_n|^{-1} \sum_{h \in F_n \setminus F_{n,\nu}} \|I_\mu(L_i \vert T^-_{n,h}(i))\|_{\nu}  \xrightarrow[n \to \infty]{} 0	&	\text{for all $1 \leq i \leq \ell$},
\end{eqnarray}
then
\begin{eqnarray}
\mathrm{P}(\Phi)	&	=						&	\frac{1}{[G:H]} \sum_{i=1}^\ell \int{\left(I_{\mu,\nu}(L_i \vert G^-_i) + \frac{1}{\ell} \varphi_K\right)}d\nu	\\
		&	 \stackrel{\nu(x)\text{-a.e.}}{=}	&	\frac{1}{[G:H]} \sum_{i=1}^\ell \E_\nu\left[I_{\mu,\nu}(L_i \vert G^-_i) + \frac{1}{\ell} \varphi_K \middle\vert \invAlg_H\right](x).
\end{eqnarray}

On the other hand, if $\ell = 1$, we also have the following converse:
\begin{equation}
|F_n|^{-1} \sum_{h \in F_n \setminus F_{n,\nu}} \|I_\mu(K \vert T^-_{n,h})\|_{\nu}  \xrightarrow[n \to \infty]{} 0
\end{equation}
if and only if
\begin{equation}
\mathrm{P}(\Phi) \stackrel{\nu(x)\text{-a.e.}}{=} \E_\nu[I_{\mu,\nu}(K \vert G^-) + \varphi_K \vert \invAlg_H](x).
\end{equation}

In particular, if $\ell = 1$ and $\nu \in \ergMeasH$, then
\begin{equation}
|F_n|^{-1} \sum_{h \in F_n \setminus F_{n,\nu}} \|I_\mu(K \vert T_{n,h}^-)\|_{\nu}  \xrightarrow[n \to \infty]{} 0 \iff \mathrm{P}(\Phi) = \int{(I_{\mu,\nu}(K \vert G^-) + \varphi_K)}d\nu.
\end{equation}
\end{theorem}

\begin{proof}
By Proposition \ref{prop2}, if $I_\mu(L_i \vert T^-_{n,h}(i)) \xrightarrow[{[\{F_n\}]}]{L^1_\nu} I_{\mu,\nu}(L_i \vert G^-_i)$, then 
\begin{equation}
|F_n|^{-1} \left[\sum_{h \in F_{n,\nu}} I_\mu(L_i \vert T^-_{n,h}(i))(h \cdot x) - \sum_{h \in F_n} I_{\mu,\nu}(L_i \vert G^-_i)(h \cdot x)\right] \xrightarrow[n \to \infty]{L^1_\nu} 0.
\end{equation}

Combining this with $|F_n|^{-1} \sum_{h \in F_n \setminus F_{n,\nu}} \|I_\mu(L_i \vert T^-_{n,h}(i))\|_{\nu}  \xrightarrow[n \to \infty]{} 0$, we obtain
\begin{equation}
|F_n|^{-1} \sum_{h \in F_n} \left[I_\mu(L_i \vert T^-_{n,h}(i)) - I_{\mu,\nu}(L_i \vert G^-_i)\right](h \cdot x) \xrightarrow[n \to \infty]{L^1_\nu} 0.
\end{equation}

Then, if we identify $I_{\mu,\nu}(L_i \vert G^-_i)$ with $f_i$, everything follows from Proposition \ref{mainProp} and the fact that if $\nu \in \ergMeasH$, then $\invAlg_H$ is trivial for $\nu$.
\end{proof}

\begin{remark}
It is easy to check that when $H = G$ and $K = \{e\}$, then we can recover from Theorem \ref{pressureRep} an analogous result to the ones that appear in \cite{1-gamarnik,1-marcus,1-briceno}.
\end{remark}

\begin{question}
Is there an example where, under the other assumptions of Theorem \ref{pressureRep}, $|F_n|^{-1} \sum_{h \in F_n \setminus F_{n,\nu}} \|I_\mu(L_i \vert T^-_{n,h}(i))\|_{\nu}  \xrightarrow[n \to \infty]{} 0$ does not hold? In other words, is this assumption redundant?
\end{question}

\section{Discussion}
\label{sec5}

\subsection{Comparison with previous results}
\label{sec51}

Let $G =\Z^d$ for some $d \in \N$ endowed with the \emph{lexicographic order} and consider the directed set $(\Theta,\preccurlyeq)$, where $\Theta = \{M \in \F(G): M \subseteq G^-\}$ and 
\begin{equation}
M_1 \preccurlyeq M_2 \iff \left[\text{for all $n \in \N$, } [-n,n]^d \cap G^- \subseteq M_1 \implies [-n,n]^d \cap G^- \subseteq M_2\right].
\end{equation}

Notice that $\Theta$ is uncountable but it has a \emph{cofinal sequence} exhausting $G^-$, namely $\{M_n\}$ for $M_n := [-n,n]^d \cap G^-$. We say that a potential $\Phi$ is {\bf nearest-neighbor} if
\begin{equation}
\Phi(M,x) \neq 0 \text{ for } x \in X \implies M = \{h,h+s_i\}
\end{equation}
for some $h \in \Z^d$ and $1 \leq i \leq d$, where $s_i$ is the $i$th canonical vector which is $0$ in every but the $i$th coordinate where is $1$. This notion can be naturally extended to any group $G$ with a \emph{generating set} $\{s_1,s_2,\dots\}$.

Now we rephrase Theorem 3.1 from \cite{1-marcus} in the language of this paper. 

\begin{theorem}[{\cite[Theorem 3.1]{1-marcus}}]
\label{thmMarcus}
Given a nonempty SFT $X \subseteq S^{\Z^d}$ and a nearest-neighbor potential $\Phi$, consider $\mu \in \gibMeas$, $\nu \in \invMeasG$, and $\{T_n\}$ a F{\o}lner sequence such that
\begin{enumerate}
\item[(A1)] $(X,\{T_n\})$ satisfies condition (D),
\item[(A2)] $\lim_{M \to G^-} \mu([x_e] \vert [x_M])$ exists uniformly over $x \in \supp(\nu)$, and
\item[(A3)] $c_{\mu,\nu} := \inf_{x \in \supp(\nu)}\inf_{M \in \Theta} \mu([x_e] \vert [x_M]) > 0$.
\end{enumerate}

Then $\mathrm{P}(\Phi) = \int{\left(I_{\mu}(e \vert G^-) + \varphi\right)}d\nu$.
\end{theorem}

Condition (A2) means that for all $\epsilon > 0$, there exists $n \in \N$ such that
\begin{equation}
\label{ineq}
\left|\mu([x_e] \vert [x_{M_1}]) - \mu([x_e] \vert [x_{M_2}])\right| < \epsilon
\end{equation}
for all $x \in \supp(\nu)$ and for all $[-n,n]^d \cap G^- \subseteq M_1,M_2 \subseteq G^-$. This is equivalent to say that for all $\epsilon > 0$ there exists $M_0 \in \Theta$ such that Eq. (\ref{ineq}) holds for all $M_1,M_2 \succcurlyeq M_0$ and $x \in \supp(\nu)$.

This theorem was a generalization of \cite[Theorem 1]{1-gamarnik}. We claim that Theorem \ref{pressureRep} is a generalization of Theorem \ref{thmMarcus} (i.e., \cite[Theorem 3.1]{1-marcus}). Indeed, $\Z^d$ is a countable discrete amenable linearly bi-ordered group if we consider $\leq$ the lexicographic order and $\{[-n,n]^d\}$ as a F{\o}lner sequence. Now, given a nonempty SFT $X$ and a nearest-neighbor (and therefore, absolutely summable) potential $\Phi$, it is always the case that $\mathcal{G}(X,\Phi) \neq \emptyset$ (see \cite{1-ruelle}) and, due to Corollary \ref{corsupp}, $\supp(\mu) = X$ for all $\mu \in \mathcal{G}(X,\Phi)$.

Given $M \in \Theta$, define $f_M: X \to \R$ as $f_M(x) := \mu([x_e] \vert [x_M])$ for $x \in X$. Notice that $f_M$ is a \emph{local function} (i.e., it depends on finitely many coordinates), it is defined everywhere on $\supp(\mu) = X$, and therefore it is continuous.

If $Y := \supp(\nu) \subseteq X$, condition (A2) implies that $\{f_M\}_{M \in \Theta}$ converges uniformly (in the net sense) on $Y$ to some function $f_\infty: Y \to \R$ and condition (A3) implies that $f_M(x) \geq c_{\mu,\nu} > 0$ for all $x \in Y$ and $M \in \Theta$. Consequently, each element in $\{-\log \left.f_M\right|_Y\}_{M \in \Theta}$ is also continuous and, due to the Mean Value Theorem, $-\log f_M$ converges uniformly on $Y$ to $-\log f_\infty$. Therefore, for any F{\o}lner sequence $\{T_n\}$, since $T_{n,h}^- = T_nh^{-1} \cap G^- \in \Theta$ for all $n \in \N$ and $h \in T_n$, it can be checked that
\begin{equation}
I_\mu(e \vert T_{n,h}^-) = -\log f_{T_{n,h}^-} \xrightarrow[{[\{T_n\}]}]{L^1_\nu} -\log f_\infty = I_{\mu}(e \vert G^-) =: I_{\mu,\nu}(e \vert G^-).
\end{equation}

On the other hand, $\|I_\mu(e \vert T_{n,h}^-)\|_{\nu} = \|-\log f_{T_{n,h}^-}\|_{\nu} \leq  \|-\log f_{T_{n,h}^-}\|_\infty \leq \log(1/c_{\mu,\nu})$, so $\sup_n \max_{h \in T_n \setminus T_{n,\nu}} \|I_\mu(e \vert T_{n,h}^-)\|_{\nu} < \infty$ and therefore,
\begin{equation}
|T_n|^{-1} \sum_{h \in T_n \setminus T_{n,\nu}} \|I_\mu(e \vert T^-_{n,h})\|_{\nu}  \xrightarrow[n \to \infty]{} 0.
\end{equation}

It can be noticed that in Theorem \ref{pressureRep} no continuity of $I_{\mu,G^-}$ is required nor uniform convergence in $\supp(\nu)$ necessary, in contrast to Theorem \ref{thmMarcus}. Another advantage is that convergence and bounds in Theorem \ref{pressureRep} are all along a particular directed set of subsets of the past and not all of them.

\subsection{Some implications and examples}
\label{examples}

In \cite{1-marcus}, it is shown that Condition (A2) from Theorem \ref{thmMarcus} holds for any nearest-neighbor potential $\Phi$ such that the corresponding Gibbs $(X,\Phi)$-specification satisfies a form of correlation decay called \emph{strong spatial mixing (SSM)}, also known as \emph{regular complete analyticity}; see \cite{2-dobrushin} for a good survey on this subject. Although the definition of strong spatial mixing is usually stated for nearest-neighbor (or \emph{finite range}) potentials or \emph{Markov random fields} (e.g., see \cite[Definition 2.11]{1-marcus}), it can be naturally extended to absolutely summable potentials (e.g., see \cite{1-laroche} and the infinite range \emph{Ising model} case). On the other hand, in \cite{1-briceno}, it is proven that Conditions (A1) and (A3) from Theorem \ref{thmMarcus} hold whenever the $\Z^d$-subshift $X$ satisfies a property called \emph{topological strong spatial mixing (TSSM)}, also introduced in \cite{1-briceno}.

These definitions and results were originally developed for $\Z^d$-subshifts. In the general group $G$ case, what we mostly have to take into consideration is that the definition of SSM and TSSM rely on a metric on $G$, which can be naturally taken to be the \emph{word metric} for a given generating set (and coincides with the graph metric in the corresponding \emph{Cayley graph}). Then, the definitions and most results can be naturally extended to $G$-subshifts on finitely generated groups $G$ and, more generally, to any closed set of \emph{colorings} of a countable graph (see \cite[Definition 4.2 and Definition 4.6]{2-briceno}).

Then, in the context of \emph{finitely generated} groups $G$ (i.e., when the generating set can be chosen to be finite), our results can be easily proven to hold for any $G$-subshift $X$ that satisfies TSSM (which, in particular, implies $X$ is an SFT) and any absolutely summable potential $\Phi$ that induces a Gibbs $(X,\Phi)$-specification satisfying SSM (which, in particular, implies there exists a unique Gibbs measure in $\gibMeas$). Consequently, based on all the discussion here and in Section \ref{sec51}, it is possible to prove the following corollary.

\begin{corollary}
\label{corPress}
Consider
\begin{itemize}
\item $G$ a countable discrete finitely generated group,
\item $(H,\leqslant)$ an amenable linearly right-ordered subgroup $H \leq G$ of finite index,
\item $K \subseteq G$ a left transversal so that $G = KH$,
\item $X$ a nonempty $G$-subshift that satisfies TSSM,
\item $\{F_n\}$ a left F{\o}lner sequence for $H$ such that $(X,\{KF_n\})$ satisfies condition (D), and
\item $\Phi$ a potential such that the Gibbs $(X,\Phi)$-specification satisfies SSM.
\end{itemize}

Then, for the unique Gibbs measure $\mu \in \gibMeas$, we have that
\begin{equation}
\mathrm{P}(\Phi) = \int{(I_\mu(K \vert G^-) + \varphi_K)}d\nu
\end{equation}
for any $\nu \in \invMeasH$.
\end{corollary}

As an application of Corollary \ref{corPress}, we have the following: In \cite{1-gamarnik}, it was shown that the pressure of the \emph{hardcore lattice gas model with activity $\lambda$} in the \emph{$d$-dimensional hypercubic lattice $\Z^d$} has a pressure representation that corresponds to a very particular case of Theorem \ref{pressureRep}. Moreover, in \cite{1-gamarnik}, it is proven that this special representation --which corresponds to the case where  $\nu$ is an atomic measure supported on a single fixed point-- holds whenever the model satisfies SSM in the Cayley graph associated to the canonical generating set of $\Z^d$. Using similar arguments and relying on Corollary \ref{corPress}, now it can be checked that, for sufficiently small $\lambda > 0$, the same result is true for the hardcore model on any Cayley graph of any finitely generated virtually orderable group $G$ and for any $\nu \in \invMeasH$. In other words, for this family of models (that consists of $G$-subshifts $X \subseteq \{0,1\}^G$ that satisfies TSSM, where $0$ is a safe symbol and $\Phi$ is a single-site potential modulated by $\lambda$), if the parameter $\lambda$ is small enough, then SSM holds and we can assert that $\mathrm{P}(\Phi) = \int{(I_\mu(K \vert G^-) + \varphi_K)}d\nu$ for any $\nu \in \invMeasH$. By small enough, we mean that $\lambda$ satisfies a criterion like, for example, the one that appears in \cite{1-weitz}, which depends exclusively on the degree of the Cayley graph. In particular, if $\nu = \delta_{0^G}$, we have that $\mathrm{P}(\Phi) = I_\mu(0^K \vert 0^{G^-})$, as in the original case from \cite{1-gamarnik} for $G = \Z^d$, but in a much more general setting. This will be further explored in future work.

\subsection{Case: $\nu = \mu$}
\label{mueqnu}

A very important case is when the measure $\nu$ coincides with the Gibbs measure $\mu$. By Theorem \ref{pinsker} and Theorem \ref{thm1}.(2), we know that
\begin{equation}
\mathrm{P}(\Phi) = h(\mu) + \int{\varphi}d\mu = \int{(I_{\mu,G^-} + \varphi)}d\mu.
\end{equation}

Now, we show that we can obtain a more general expression for $\mathrm{P}(\Phi)$, which resembles the work in \cite{1-helvik}, but in the context of amenable virtually orderable groups. In the context of Theorem \ref{pressureRep}, if $\nu = \mu$, we first claim that
\begin{equation}
\mathrm{P}(\Phi) = \frac{1}{[G:H]} \sum_{i=1}^\ell \int{\left(I_\mu(L_i \vert G^-_i) + \frac{1}{\ell} \varphi_K\right)}d\mu,
\end{equation}
because the assumptions
\begin{enumerate}
\item $I_\mu(L_i \vert T^-_{n,h}(i)) \xrightarrow[{[\{F_n\}]}]{L^1_\mu} I_{\mu,\mu}(L_i \vert G^-_i)$ and
\item $|F_n|^{-1} \sum_{h \in F_n \setminus F_{n,\mu}} \|I_\mu(L_i \vert T^-_{n,h}(i))\|_{\mu}  \xrightarrow[n \to \infty]{} 0$
\end{enumerate}
follow for free. It suffices to prove that
\begin{eqnarray}
I_{\mu,\mu}(L_i \vert G^-_i)(x) \stackrel{\mu(x)\text{-a.e.}}{=}  I_{\mu}(L_i \vert G^-_i)(x)	&	\text{for all $1 \leq i \leq \ell$.}
\end{eqnarray}

Here, by $I_{\mu,\mu}(L_i \vert G^-_i)$, we mean the limit $I_{\mu}(L_i \vert T^-_{n,h}(i)) \xrightarrow[{[\{F_n\}]}]{L^1_\mu} I_{\mu,\mu}(L_i \vert G^-_i)$ and $I _{\mu,G^-}$ refers to the information function of $\alpha^{L_i}$ conditioned on $\B_{G^-}$. The existence of the limit and the equality follow from \cite[Lemma 2.3]{1-krengel} adapted to our case.

\begin{lemma}[{\cite[Lemma 2.3]{1-krengel}}]
\label{kreng}
Let $\Theta$ be a denumerable directed set and $\{\B_\theta: \theta \in \Theta\}$ an increasing family of sub-$\sigma$-algebras of $\B_G$. If $M \in \F(G)$ and $\C_\infty = \bigvee_\theta \B_\theta$ is the $\sigma$-algebra generated by all $\B_\theta$, then $I_\mu(\alpha^M \vert \C_\theta) \xrightarrow[\Theta]{L^1_\mu} I_\mu(\alpha^M \vert \C_\infty)$.
\end{lemma}

In our case, $\Theta = [\{F_n\}]$, $\C_{(n,h)} = \B_{T_{n,h}^-(i)}$, and $\C_\infty = \B_{G^-_i}$, so $\B_{G^-_i} = \bigvee_{(n,h) \in  [\{F_n\}]} \B_{T_{n,h}^-(i)}$, and $M = L_i$. Then, it is guaranteed that
\begin{equation}
I_{\mu}(L_i \vert T^-_{n,h}(i)) \xrightarrow[{[\{F_n\}]}]{L^1_\mu} I_{\mu}(L_i \vert G^-_i).
\end{equation}

On the other hand, it can be checked that
\begin{equation}
|F_n|^{-1} \sum_{h \in F_n \setminus F_{n,\mu}} \|I_\mu(L_i \vert T^-_{n,h}(i))\|_{\mu}  \xrightarrow[n \to \infty]{} 0,
\end{equation}
since $\|I_\mu(L_i \vert T^-_{n,h}(i))\|_{\mu} \leq H_\mu(\alpha^{L_i}) < \infty$ and $\frac{|F_n \setminus F_{n,\mu}|}{|F_n|} \to 0$. Therefore, Theorem \ref{pressureRep} applies.

\subsection{Entropy of lattice systems}
\label{helviklindgren}

To finish, notice that the decomposition given in Eq. (\ref{formula}), Lemma \ref{kreng}, and the discussion above hold for arbitrary $\nu \in \invMeasG$, not necessarily a Gibbs measure. Reusing these ideas, we can conclude that
\begin{equation}
\label{helvik1}
h(\nu) = \frac{1}{[G:H]} \sum_{i=1}^\ell \int{I_\nu(L_i \vert G^-_i)}d\nu = \frac{1}{[G:H]} \sum_{i=1}^\ell H_\nu(\alpha^{L_i} \vert \B_{G^-_i}).
\end{equation}

In \cite{1-helvik}, Helvik and Lindgren explored ways to represent measure-theoretical entropy of lattice systems as sums of conditional entropies. Formally, they considered an arbitrary \emph{$d$-dimensional point lattice} $G$ (see \cite{1-helvik} for all relevant definitions) and a \emph{tiling} $(K,H)$ consisting of a finite subset $K \subseteq G$ with $e \in K$ and a $d$-dimensional subgroup $H \leq G$ satisfying
\begin{enumerate}
\item $K + H = G$ and
\item $(K + h_1) \cap (K + h_2) = \emptyset$ for all $h_1,h_2 \in H$, $h_1 \neq h_2$.
\end{enumerate}

Then, they had the following theorem (where we have combined notation and terminology from their work and the present one in order to help the reader to compare both of them).
\begin{theorem}[{\cite[Theorem 1]{1-helvik}}]
Let $G$ be a $d$-dimensional point lattice, $\nu \in \mathcal{M}_G(S^G)$, and $(K,H)$ be a tiling of $G$. Partition $K$ into $\ell$ nonempty pairwise disjoint sets $L_1,\dots,L_\ell$, with $1 \leq \ell \leq |K|$, and define $K_i = \bigcup_{j=1}^{i} L_j$. Then,
\begin{equation}
\label{helvik2}
h(\nu) = \frac{1}{|K|} \sum_{i=1}^\ell H_\nu\left(\alpha^{L_i} \vert \B_{(L_i + H^-) \cup (K_{i-1} + H)}\right).
\end{equation}
\end{theorem}

By comparing Eq. (\ref{helvik1}) and Eq. (\ref{helvik2}), it is easy to see that, after identifying and renaming some terms, the two equations are identical. This gives a generalization of the result from Helvik and Lindgren and, in particular, covers both the abelian and non-abelian (virtually orderable) cases.

\section*{Acknowledgements}

I would like to thank Brian Marcus for his relevant suggestions and comments, Fabio Martinelli for his help regarding complete analytical interactions, Ricardo G\'omez and Siamak Taati for pointing out fundamental references, and Crist\'obal Rivas for his orientation regarding the study of orderable groups. I would also like to acknowledge Rodrigo Bissacot, Lu\'isa Borsato, and the anonymous referee for their valuable comments that helped to improve the presentation of this work.

The author acknowledges the support of ERC Starting Grants 678520 and 676970.

\bibliographystyle{ijmart}
\bibliography{ijmsample}

\end{document}